\documentclass[12pt,twoside]{amsart}

\usepackage{amsmath,amsthm,amscd,amssymb,mathrsfs,graphicx,amsfonts,mathrsfs}
\usepackage{amsfonts}
\usepackage{amssymb,enumerate}
\usepackage{amsthm}
\usepackage[all]{xy}
\usepackage{hyperref}
\allowdisplaybreaks[4] \footskip=12pt
\renewcommand{\uppercasenonmath}[1]{}

\numberwithin{equation}{section} \theoremstyle{plain}
\newtheorem*{thm*}{Main Theorem}
\newtheorem{thm}{Theorem}[section]
\newtheorem{cor}[thm]{Corollary}
\newtheorem*{cor*}{Corollary}
\newtheorem{lem}[thm]{Lemma}
\newtheorem*{lem*}{Lemma}

\newtheorem*{fact*}{Fact}

\newtheorem*{nota*}{Notation}

\newtheorem*{prop*}{Proposition}
\newtheorem{rem}[thm]{Remark}
\newtheorem*{rem*}{Remark}

\newtheorem*{observation*}{Observation}

\newtheorem*{exa*}{Example}
\newtheorem{df}[thm]{Definition}
\newtheorem*{df*}{Definition}

\newtheorem*{conj*}{Conjecture}

\renewcommand{\geq}{\geqslant}
\renewcommand{\leq}{\leqslant}

\begin{document}
\begin{center}
{\large  \bf The local-global principle for the artinianness dimensions}

\vspace{0.3cm} Jingwen Shen, Xiaoyan Yang \\
Zhejiang University of Science and Technology, Hangzhou 310023, China\\
E-mails: shenjw0609@163.com,  yangxiaoyanzj@outlook.com
\end{center}
\bigskip
\centerline { \bf  Abstract} Let $R$ be a commutative noetherian ring and $\mathfrak{a}$ an ideal of $R$. The goal of this paper is to establish the local-global principle for the artinianness dimension $r_{\mathfrak{a}}(M)$, where $r_{\mathfrak{a}}(M)$ is the smallest integer such that the local homology module of $M$ is not artinian. For an artinian $R$-module $M$ with the set $\mathrm{Coass}_{R}\mathrm{H}_{r_{\mathfrak{a}}(M)}^{\mathfrak{a}}(M)$ finite, we show that $r_{\mathfrak{a}}(M)=\mathrm{inf}\{r_{\mathfrak{a}R_{\mathfrak{p}}}(\mathrm{Hom}_{R}(R_{\mathfrak{p}},M)) \hspace{0.03cm}|\hspace{0.03cm}\mathfrak{p}\in \mathrm{Spec}R\}$. And the class of all modules $N$ such that $\mathrm{Coass}_{R}N$ is finite is studied.
\leftskip10truemm \rightskip10truemm \noindent \\
\vbox to 0.3cm{}\\
{\it Key Words:} artinianness dimension; local homology\\
{\it 2020 Mathematics Subject Classification:} 13C15; 13J10

\leftskip0truemm \rightskip0truemm
\bigskip

\section{\bf Introduction and Preliminaries}
\renewcommand{\thethm}{\Alph{thm}}
Throughout this paper, let $R$ be a commutative noetherian ring, $\mathfrak{a}$ an ideal of $R$. Denote $\mathrm{Spec}R$ the set of primes ideals of $R$,  $\mathrm{V}(\mathfrak{a})=\{\mathfrak{p}\in \mathrm{Spec}R\hspace{0.03cm}|\hspace{0.03cm}\mathfrak{a}\subseteq \mathfrak{p}\}$. Fix $\mathfrak{p}\in \mathrm{Spec}R$, $M_{\mathfrak{p}}$ denote the localization of $R$-module $M$ at $\mathfrak{p}$, the colocalization $\mathrm{Hom}_{R}(R_{\mathfrak{p}},M)$ of $M$ at $\mathfrak{p}$ briefly represented by $_{\mathfrak{p}}M$.

For an $R$-module $M$, the $i$th \emph{local cohomology module} of $M$ with respect to $\mathfrak{a}$ is defined as
\begin{center}$\begin{aligned}
\mathrm{H}^{i}_{\mathfrak{a}}(M)=\underrightarrow{\text{lim}}\mathrm{Ext}^{i}_{R}(R/\mathfrak{a}^{t},M).
\end{aligned}$\end{center}
If $(R,\mathfrak{m})$ is a local ring and $N$ a non-zero finitely generated $R$-module of dimension $d>0$, then $\mathrm{H}^{0}_{\mathfrak{m}}(N)$ is finitely generated, while $\mathrm{H}^{d}_{\mathfrak{m}}(N)$ is not finitely generated and $\mathrm{H}^{i}_{\mathfrak{m}}(N)=0$ for $i> d$ by \cite[Theorem 6.1.2, Corollary 7.3.3]{B1}. It becomes of interest to identify the least integer $i$ such that $\mathrm{H}^{i}_{\mathfrak{m}}(N)$ is not finitely generated. This integer is called the finiteness dimension of $N$ with respect to $\mathfrak{m}$. More generally, the finiteness dimension of $N$ relative to $\mathfrak{a}$ is defined as
\begin{center}
$f_{\mathfrak{a}}(N):= \mathrm{inf}\{i\geq 0 \hspace{0.03cm}|\hspace{0.03cm}
\mathrm{H}_{\mathfrak{a}}^{i}(N)~\mathrm{is~not~finitely~generated}\}$.
\end{center}
Faltings \cite[Satz 1]{F0} proved that for a positive integer $s$, the $R_{\mathfrak{p}}$-module $\mathrm{H}_{\mathfrak{a}R_{\mathfrak{p}}}^{i}(N_{\mathfrak{p}})$ is finitely generated for $i< s$ and all $\mathfrak{p}\in \mathrm{Spec}R$ if and only if the $R$-module $\mathrm{H}_{\mathfrak{a}}^{i}(N)$ is finitely generated for $i< s$. An immediate consequence of the Faltings' result is
\begin{center}$\begin{aligned}
f_{\mathfrak{a}}(N)
&= \mathrm{inf}\{f_{\mathfrak{a}R_{\mathfrak{p}}}(N_{\mathfrak{p}}) \hspace{0.03cm}|\hspace{0.03cm}
\mathfrak{p}\in \mathrm{Spec}R\},
\end{aligned}$\end{center}
which is called local-global principle for finiteness dimension.

Local homology as a duality of local cohomoloy was initiated by Matlis \cite{M4} in 1974. Denote by $\Lambda_{\mathfrak{a}}(M)=\underleftarrow{\text{lim}}M/\mathfrak{a}^{t}M$ the $\mathfrak{a}$-adic completion of $M$ and recall the $i$th \emph{local homology module} of $M$ defined in \cite{CN0} is
\begin{center}$\begin{aligned}
\mathrm{H}_{i}^{\mathfrak{a}}(M)=\underleftarrow{\text{lim}}\mathrm{Tor}_{i}^{R}(R/\mathfrak{a}^{t},M).
\end{aligned}$\end{center}
Cuong and Nam \cite{CN1} proved that the local homology defined in this way behaves similar properties to local cohomology in the category of linearly compact $R$-modules, specially, in the category of artinian $R$-modules. For example, $\mathrm{H}_{0}^{\mathfrak{a}}(M)$ is artinian when $M$ is an artinian $R$-module and $\mathrm{H}_{i}^{\mathfrak{a}}(M)=0$ for $i> \mathrm{mag}_{R}M$, where $\mathrm{mag}_{R}M$ is the magnitude of $M$ defined by Yassemi \cite{Y2}.

Motivated by the finiteness dimension, the artinianness dimension of $M$ with respect to $\mathfrak{a}$ is defined as
\begin{center}$\begin{aligned}
r_{\mathfrak{a}}(M):
&= \mathrm{inf}\{i\geq 0 \hspace{0.03cm}|\hspace{0.03cm}
\mathrm{H}^{\mathfrak{a}}_{i}(M)~\mathrm{is~not~artinian}\}.
\end{aligned}$\end{center}
The aim of this article is to build the local-global principle for the artinianness dimension.
More precisely, we prove the following theorem (see Theorem \ref{thm:3.6}).

\begin{thm}\label{thm:1.1}
Let $M$ be an artinian $R$-module such that the set $\mathrm{Coass}_{R}\mathrm{H}_{r_{\mathfrak{a}}(M)}^{\mathfrak{a}}(M)$ is finite. One has an equality
\begin{center}
$r_{\mathfrak{a}}(M)=
\mathrm{inf}\{r_{\mathfrak{a}R_{\mathfrak{p}}}(\mathrm{Hom}_{R}(R_{\mathfrak{p}},M)) \hspace{0.03cm}|\hspace{0.03cm}\mathfrak{p}\in \mathrm{Spec}R\}.$
\end{center}
\end{thm}

In section 3, we study a class $FM_{\leq n}$ of $R$-modules, show that  \begin{center}$(\mathrm{Coass}_{R}\mathrm{H}_{\mathrm{g}^{\mathfrak{a}}_{n}(M)}^{\mathfrak{a}}(M))_{> n}= \mathrm{inf}\{\mathfrak{p}\in \mathrm{Coass}_{R}\mathrm{H}_{\mathrm{g}^{\mathfrak{a}}_{n}(M)}^{\mathfrak{a}}(M) \hspace{0.03cm}|\hspace{0.03cm}
\mathrm{dim}R/\mathfrak{p}> n\}$\end{center} is finite, where $\mathrm{g}^{\mathfrak{a}}_{n}(M):=\mathrm{inf}\{i\geq0\hspace{0.03cm}|\hspace{0.03cm}
\mathrm{H}^{\mathfrak{a}}_{i}(M)\notin FM_{\leq n}\}$. Moreover, we find that all semi-discrete linearly compact $\mathfrak{a}$-coartinian $R$-modules satisfy the above equality.

\vspace{2mm}
Now we list some notions which will need later.

{\bf Coassociated prime and magnitude.} The \emph{support} of an $R$-module $M$, denoted by $\mathrm{Supp}_{R}M$, is the set of prime ideals of $\mathfrak{p}$ such that there is a cyclic submodule $N$ of $M$ with $\mathrm{Ann}_{R}N\subseteq \mathfrak{p}$.
The \emph{(Krull) dimension} of $M$ is
\begin{center}
$\mathrm{dim}_{R}M=\mathrm{sup}\{\mathrm{dim}R/\mathfrak{p}\hspace{0.03cm}|\hspace{0.03cm}\mathfrak{p}\in \mathrm{Supp}_{R}M\}$.
\end{center}
If $M=0$, then write $\mathrm{dim}_{R}M=-\infty$.

Yassemi \cite{Y1} defined the cocyclic modules. An $R$-module $L$ is \emph{cocyclic} if $L$ is a submodule of $E(R/\mathfrak{m})$ for some maximal ideal $\mathfrak{m}$. The \emph{coassociated prime} of $M$, denoted by $\mathrm{Coass}_{R}M$, is the set of prime ideals $\mathfrak{p}$ such that there is a cocyclic homomorphic image $L$ of $M$ with $\mathfrak{p}=\mathrm{Ann}_{R}L$. If the equal condition is reduced to include, then the set of primes ideals is called \emph{cosupport} of $M$ and denote by $\mathrm{Cosupp}_{R}M$.
Yassemi \cite{Y2} then defined \emph{magnitude} of modules, a dual concept of dimension, as
\begin{center}$\begin{aligned}
\mathrm{mag}_{R}M=\mathrm{sup}\{\mathrm{dim}R/\mathfrak{p}\hspace{0.03cm}|\hspace{0.03cm}\mathfrak{p}\in \mathrm{Cosupp}_{R}M\}.
\end{aligned}$\end{center}
Write $\mathrm{mag}_{R}M=-\infty$ if $M=0$.

Following \cite{M}, a topological $R$-module $M$ is said to be \textit{linearly topologized} if it has a base of neighborhoods of the zero element $\mathcal{M}$ consisting of submodules; $M$ is called \textit{Hausdorff} if the intersection of all the neighborhoods of the zero element is $0$. A Hausdorff linearly topologized $R$-module $M$ is said to be \textit{linearly compact} if $\mathcal{F}$ is a family of closed cosets (i.e., cosets of closed submodules) in $M$ which has the finite intersection property, then the cosets in $\mathcal{F}$ have a non-empty intersection. It should be noted that an artinian $R$-module is linearly compact. A Hausdorff linearly topologized $R$-module $M$ is called \emph{semi-discrete} if every submodule of $M$ is closed. The class of semi-discrete linearly compact modules is very large, it contains many important classes of modules such as the class of artinain modules, the class of finitely generated modules over a complete ring.

\bigskip
\section{\bf Local-global principle for the artinianness }
\renewcommand{\thethm}{\arabic{section}.\arabic{thm}}

Denote $\mathcal{S}$ a \emph{Serre subcategory} of the category of $R$-modules which means that it is closed under taking submodules, quotients and extensions. The classes of finitely generated, artinian $R$-modules are examples of Serre subcategories. In this section, the proof of Theorem \ref{thm:1.1} is provided. We begin with the following lemmas.

\begin{lem}\label{lem:2.7}
Let $s$ be a non-negative integer and $M$ a linearly compact $R$-module.

$\mathrm{(1)}$ If $\mathrm{Tor}_{s}^{R}(R/\mathfrak{a},M)\in \mathcal{S}$ and
$\mathrm{Tor}_{j}^{R}(R/\mathfrak{a},\mathrm{H}_{i}^{\mathfrak{a}}(M))\in \mathcal{S}$ for $i< s$ and $j\geq 0$, then $R/\mathfrak{a}\otimes_{R}\mathrm{H}_{s}^{\mathfrak{a}}(M)\in \mathcal{S}$.

$\mathrm{(2)}$ If $\mathrm{Tor}_{s+1}^{R}(R/\mathfrak{a},M)\in \mathcal{S}$ and $\mathrm{Tor}_{j}^{R}(R/\mathfrak{a},\mathrm{H}_{i}^{\mathfrak{a}}(M))\in \mathcal{S}$ for $i< s$ and $j\geq 0$, then $\mathrm{Tor}_{1}^{R}(R/\mathfrak{a},\mathrm{H}_{s}^{\mathfrak{a}}(M))\in \mathcal{S}$.
\end{lem}

\begin{proof}
We just prove $\mathrm{(1)}$ since $\mathrm{(2)}$ follows by a similar argument. 

Consider the spectral sequence
\begin{center}
$\xymatrix@C=10pt@R=5pt{
E^2_{p,q}=\mathrm{Tor}_{p}^R(R/\mathfrak{a},\mathrm{H}_{q}^\mathfrak{a}(M))\ar@{=>}[r]_{\ \ \ \ \ p}& \mathrm{Tor}_{p+q}^R(R/\mathfrak{a},M).}$
\end{center}
By assumption, $E^{2}_{p,q}\in \mathcal{S}$ for $p\geq 0$ and $q\leq s-1$. There is a finite filtration
\begin{center}
$0=U^{-1}\subseteq U^{0}\subseteq \cdots\subseteq U^{s-1}\subseteq U^{s}=\mathrm{Tor}_{s}^{R}(R/\mathfrak{a},M)$,
\end{center}
such that $U^{p}/U^{p-1}\cong E^{\infty}_{p,s-p}$ for every $0\leq p\leq s$. As $\mathrm{Tor}_{s}^{R}(R/\mathfrak{a},M)\in \mathcal{S}$, it follows that $E^{\infty}_{0,s}\cong U^{0}/U^{-1}\in \mathcal{S}$. Let $r\geq 2$, consider the differentials
$$\xymatrix{E^{r}_{r,s-r+1~} \ar[r]^-{d^{r}_{r,s-r+1}} & E^{r}_{0,s} \ar[r]^-{d^{r}_{0,s}} & E^{r}_{-r,s+r-1}=0.}$$
Since $s-r+1\leq s-1$ and $E^{r}_{r,s-r+1}$ is a subquotient of $E^{2}_{r,s-r+1}$, it follows that $E^{r}_{r,s-r+1}\in \mathcal{S}$, consequently $\mathrm{im}d^{r}_{r,s-r+1}\in \mathcal{S}$. Thus we obtain a short exact sequence
\begin{center}
$0\rightarrow \mathrm{im}d^{r}_{r,s-r+1}\rightarrow E^{r}_{0,s}\rightarrow E^{r+1}_{0,s}\rightarrow 0$.
\end{center}
Note that there is an integer $r_{0}\geq 2$ such that $E^{r+1}_{0,s}\cong  E^{\infty}_{0,s}\in \mathcal{S}$ for $r\geq r_{0}$. It follows that $E^{r_{0}+1}_{0,s}\in \mathcal{S}$. Hence the above exact sequence implies that $E^{r_{0}}_{0,s}\in \mathcal{S}$. Using the exact sequence inductively, we have $R/\mathfrak{a}\otimes_{R}\mathrm{H}_{s}^{\mathfrak{a}}(M)\cong E^{2}_{0,s}\in \mathcal{S}$.
\end{proof}

The next lemma provides a characterization of artinianness of local homology modules, which is a generalization of \cite[Proposition 4.7]{CN0}.

\begin{lem}\label{lem:3.1}
Let $M$ be a linearly compact $R$-module. Suppose that $t$ is a non-negative integer such that $\mathrm{Tor}_{i}^{R}(R/\mathfrak{a},M)$ is artinian for $i< t$. Then the following are equivalent:

$\mathrm{(1)}$ $\mathrm{H}_{i}^{\mathfrak{a}}(M)$ is artinian for $i< t$.

$\mathrm{(2)}$ $\mathfrak{a}\subseteq \sqrt{\mathrm{Ann}_{R}\mathrm{H}_{i}^{\mathfrak{a}}(M)}$ for $i< t$.
\end{lem}

\begin{proof}
$\mathrm{(1)}\Rightarrow \mathrm{(2)}$ Assume that $\mathrm{H}_{i}^{\mathfrak{a}}(M)$ is artinian for $i< t$. Then $\mathrm{Cosupp}_{R}\mathrm{H}_{i}^{\mathfrak{a}}(M)\subseteq \mathrm{V}(\mathrm{Ann}_{R}\mathrm{H}_{i}^{\mathfrak{a}}(M))\subseteq \mathrm{V}(\mathfrak{a})$ for $i< t$. Thus we have $\mathfrak{a}\subseteq \sqrt{\mathfrak{a}}\subseteq \sqrt{\mathrm{Ann}_{R}\mathrm{H}_{i}^{\mathfrak{a}}(M)}$ for $i< t$.

$\mathrm{(2)}\Rightarrow \mathrm{(1)}$ Using induction on $t$. If $t=1$, then $\mathfrak{a}\subseteq \sqrt{\mathrm{Ann}_{R}\mathrm{H}_{0}^{\mathfrak{a}}(M)}$, so that there is an integer $l\geq 1$ such that $\mathfrak{a}^{l}\mathrm{H}_{0}^{\mathfrak{a}}(M)=0$ and $\mathrm{H}_{0}^{\mathfrak{a}}(M)/\mathfrak{a}^{l}\mathrm{H}_{0}^{\mathfrak{a}}(M)\cong \mathrm{H}_{0}^{\mathfrak{a}}(M)$. Since $M/\mathfrak{a}M$ is artinian, we get $M/\mathfrak{a}^{l}M$ is artinian. Hence $\mathrm{H}_{0}^{\mathfrak{a}}(M)$ is artinian by the epimorphism $M/\mathfrak{a}^{l}M \rightarrow \mathrm{H}_{0}^{\mathfrak{a}}(M)$. Now suppose inductively that $t> 1$ and we have established the result for smaller values of $t-1$. By assumption that $\mathrm{H}_{i}^{\mathfrak{a}}(M)$ is artinian for $i=0,1,\ldots,t-2$ and it remains to prove that $\mathrm{H}_{t-1}^{\mathfrak{a}}(M)$ is artinian. Since $\mathrm{Tor}_{t-1}^{R}(R/\mathfrak{a},M)$ is artinian and $\mathrm{Tor}_{t-1}^{R}(R/\mathfrak{a},\mathrm{H}_{i}^{\mathfrak{a}}(M))$ is artinian for
$i< t-1$, $R/\mathfrak{a}\otimes_{R}\mathrm{H}_{t-1}^{\mathfrak{a}}(M)$ is artinian by Lemma \ref{lem:2.7}. On the other hand, there exists $s> 0$ such that $\mathfrak{a}^{s}\mathrm{H}_{t-1}^{\mathfrak{a}}(M)=0$ by assumption. Hence $\mathrm{H}_{t-1}^{\mathfrak{a}}(M)\cong \mathrm{H}_{t-1}^{\mathfrak{a}}(M)/\mathfrak{a}^{s}\mathrm{H}_{t-1}^{\mathfrak{a}}(M)$ is artinian. This completes the inductive step.
\end{proof}

Lemma \ref{lem:3.1} provides some motivation for the following definition.

\begin{df}\label{df:3.0} Let $M$ be a linearly compact $R$-module such that $\mathrm{Tor}_{i}^{R}(R/\mathfrak{a},M)$ is artinian for every integer $i$. The artinianness dimension of $M$ with respect to $\mathfrak{a}$ is
\begin{center}$\begin{aligned}
r_{\mathfrak{a}}(M):
&= \mathrm{inf}\{i> 0 \hspace{0.03cm}|\hspace{0.03cm}
\mathrm{H}^{\mathfrak{a}}_{i}(M)~\mathrm{is~not~artinian}\}\\
&= \mathrm{inf}\{i> 0 \hspace{0.03cm}|\hspace{0.03cm}
\mathfrak{a}\nsubseteq \sqrt{\mathrm{Ann}_{R}\mathrm{H}_{i}^{\mathfrak{a}}(M)}\}.
\end{aligned}$\end{center}
Note that $r_{\mathfrak{a}}(M)$ is either a positive integer or $\infty$.
\end{df}

In the situation of the above definition, it is reasonable to regard the condition that $\mathfrak{a}\subseteq \sqrt{\mathrm{Ann}_{R}\mathrm{H}_{i}^{\mathfrak{a}}(M)}$  as asserting that $\mathrm{H}_{i}^{\mathfrak{a}}(M)$ is `small'  in a sense, because if
this condition holds for all $i$ less than some positive integer $t$, then $\mathrm{H}_{i}^{\mathfrak{a}}(M)$ is artinian for all $i< t$ (by Lemma \ref{lem:3.1}). However, sometimes it is more realistic to hope for a weaker condition than `$\mathfrak{a}\subseteq \sqrt{\mathrm{Ann}_{R}\mathrm{H}_{i}^{\mathfrak{a}}(M)}$': we give another ideal $\mathfrak{b}$ of $R$ with $\mathfrak{b}\subseteq \mathfrak{a}$, think of $\mathrm{H}_{i}^{\mathfrak{a}}(M)$ as being `small' relative to $\mathfrak{b}$ if $\mathfrak{b}\subseteq \sqrt{\mathrm{Ann}_{R}\mathrm{H}_{i}^{\mathfrak{a}}(M)}$.

\begin{df}\label{df:3.2}
Let $\mathfrak{a}, \mathfrak{b}$ be two ideals of $R$ with $\mathfrak{b}\subseteq \mathfrak{a}$, $M$ a linearly compact $R$-module such that $\mathrm{Tor}_{i}^{R}(R/\mathfrak{a},M)$ is artinian for every integer $i$. The $\mathfrak{b}$-artinianness dimension of $M$ relative to $\mathfrak{a}$ is defined as
\begin{center}$\begin{aligned}
r_{\mathfrak{a}}^{\mathfrak{b}}(M):
&= \mathrm{inf}\{i> 0 \hspace{0.03cm}|\hspace{0.03cm}
\mathfrak{b}\nsubseteq \sqrt{\mathrm{Ann}_{R}\mathrm{H}_{i}^{\mathfrak{a}}(M)}\}\\
&= \mathrm{inf}\{i> 0 \hspace{0.03cm}|\hspace{0.03cm}
\mathrm{mag}_{R}\mathfrak{b}^{t}\mathrm{H}_{i}^{\mathfrak{a}}(M)\geq 0~\mathrm{for~all}~t\in \mathbb{N}\}.
\end{aligned}$\end{center}
\end{df}

\begin{rem}\label{rem:3.0}\rm
$\mathrm{(1)}$ Note that $r_{\mathfrak{a}}^{\mathfrak{b}}(M)$ is either a positive integer or $\infty$ because $\mathfrak{b}\subseteq \mathfrak{a}\subseteq \sqrt{\mathrm{Ann}_{R}\mathrm{H}_{0}^{\mathfrak{a}}(M)}$.

$\mathrm{(2)}$ It is easy to see that $r_{\mathfrak{a}}^{\mathfrak{a}}(M)=r_{\mathfrak{a}}(M)$.
\end{rem}

Suppose that $n$ is a  non-negative integer. Set
\begin{center}$r_{\mathfrak{a}}^{\mathfrak{b}}(M)_{n}:
=\mathrm{inf}\{i> 0 \hspace{0.03cm}|\hspace{0.03cm}
\mathrm{mag}_{R}\mathfrak{b}^{t}\mathrm{H}_{i}^{\mathfrak{a}}(M)\geq n~\mathrm{for~all}~t\in \mathbb{N}\}$,\end{center}
\begin{center}$\widetilde{r}_{\mathfrak{a}}^{\mathfrak{b}}(M)^{n}:
=\mathrm{inf}\{r_{\mathfrak{a}R_{\mathfrak{p}}}^{\mathfrak{b}R_{\mathfrak{p}}}(_{\mathfrak{p}}M) \hspace{0.03cm}|\hspace{0.03cm}
\mathfrak{p}\in \mathrm{Spec}R~\mathrm{and}~\mathrm{dim}R/\mathfrak{p}\geq n\}$.\end{center} Next we investigate the relationship between $r_{\mathfrak{a}}^{\mathfrak{b}}(M)_{n}$ and $\widetilde{r}_{\mathfrak{a}}^{\mathfrak{b}}(M)^{n}$.

\begin{lem}\label{lem:3.4}
Let $\mathfrak{b}\subseteq \mathfrak{a}$ be ideals of $R$,  $M$ a linearly compact $R$-module so that $\mathrm{Tor}_{i}^{R}(R/\mathfrak{a},M)$ is artinian for all $i$. Then for every non-negative integer $n$,
\begin{center}
$r_{\mathfrak{a}}^{\mathfrak{b}}(M)_{n}\leq \widetilde{r}_{\mathfrak{a}}^{\mathfrak{b}}(M)^{n}$.
\end{center}
\end{lem}
\begin{proof}
Put $s=\widetilde{r}_{\mathfrak{a}}^{\mathfrak{b}}(M)^{n}$ and assume that $s< r_{\mathfrak{a}}^{\mathfrak{b}}(M)_{n}$. There is an integer $t$ such that $\mathrm{mag}_{R}\mathfrak{b}^{t}\mathrm{H}_{s}^{\mathfrak{a}}(M)< n$. By \cite[Corollary 2.16]{Y1}, for any $\mathfrak{p}\in \mathrm{Spec}R$ with $\mathrm{dim}R/\mathfrak{p}\geq n$, we have
\begin{center}
$_{\mathfrak{p}}(\mathfrak{b}^{t}\mathrm{H}_{s}^{\mathfrak{a}}(M))=0$.
\end{center}
And as $M$ is linearly compact, it follows from \cite[Theorem 3.6]{CN2} that $(\mathfrak{b}R_{\mathfrak{p}})^{t}\mathrm{H}_{s}^{\mathfrak{a}R_{\mathfrak{p}}}(_{\mathfrak{p}}M)=0$. Thus $\widetilde{r}_{\mathfrak{a}}^{\mathfrak{b}}(M)^{n}> s$, which is a contradiction.
\end{proof}

Let $T$ be a subset of $\mathrm{Spec}R$. Put
\begin{center}
$T_{> n}:=\{\mathfrak{p}\in T\hspace{0.03cm}|\hspace{0.03cm}\mathrm{dim}R/\mathfrak{p}> n\}$.
\end{center}

\begin{lem}\label{lem:3.5}
Let $\mathfrak{b}\subseteq \mathfrak{a}$ be ideals of $R$, $M$ a linearly compact $R$-module so that $\mathrm{Tor}_{i}^{R}(R/\mathfrak{a},M)$ is artinian for every integer $i$. If $(\mathrm{Coass}_{R}\mathrm{H}_{r_{\mathfrak{a}}^{\mathfrak{b}}(M)_{n}}^{\mathfrak{a}}(M))_{\geq n}$ is finite for any non-negative integer $n$, then
\begin{center}
$r_{\mathfrak{a}}^{\mathfrak{b}}(M)_{n}=\widetilde{r}_{\mathfrak{a}}^{\mathfrak{b}}(M)^{n}$.
\end{center}
\end{lem}
\begin{proof}
Put $s=r_{\mathfrak{a}}^{\mathfrak{b}}(M)_{n}$ and assume that
\begin{center}
$(\mathrm{Coass}_{R}\mathrm{H}_{s}^{\mathfrak{a}}(M))_{\geq n}=\{\mathfrak{p}_{1},\ldots,\mathfrak{p}_{k}\}$.
\end{center}
It is enough to show that $\widetilde{r}_{\mathfrak{a}}^{\mathfrak{b}}(M)^{n}\leq s$ by Lemma \ref{lem:3.4}. Suppose on the contrary that $s< \widetilde{r}_{\mathfrak{a}}^{\mathfrak{b}}(M)^{n}$. It implies that $s< r_{\mathfrak{a}R_{\mathfrak{p}_{i}}}^{\mathfrak{b}R_{\mathfrak{p}_{i}}}(_{\mathfrak{p}_{i}}M)$ for $1\leq i\leq k$. Hence there exists $l_{i}$ such that
\begin{center}
$(\mathfrak{b}R_{\mathfrak{p}_{i}})^{l_{i}}
\mathrm{H}_{s}^{\mathfrak{a}R_{\mathfrak{p}_{i}}}(_{\mathfrak{p}_{i}}M)=0$.
\end{center}
Set $l=\mathrm{max}\{l_{1},\ldots,l_{k}\}$. Then for $1\leq i\leq k$, one has
\begin{center}
$(\mathfrak{b}R_{\mathfrak{p}_{i}})^{l}
\mathrm{H}_{s}^{\mathfrak{a}R_{\mathfrak{p}_{i}}}(_{\mathfrak{p}_{i}}M)=0$.
\end{center}
As $M$ is linearly compact, it follows from \cite[Theorem 3.6]{CN2} that $_{\mathfrak{p}_{i}}(\mathfrak{b}^{l}\mathrm{H}_{s}^{\mathfrak{a}}(M))=0$. Hence $\mathrm{mag}_{R}\mathfrak{b}^{l}\mathrm{H}_{s}^{\mathfrak{a}}(M)< n$ by \cite[Theorem 3.8]{N}. This means that $s< r_{\mathfrak{a}}^{\mathfrak{b}}(M)_{n}$, which is a contradiction.
\end{proof}

Set $n=0$ in Lemma \ref{lem:3.5} and use the fact that $r_{\mathfrak{a}}^{\mathfrak{b}}(M)=r_{\mathfrak{a}}^{\mathfrak{b}}(M)_{0}$, we obtain the local-global principle for the artinianness dimension.

\begin{thm}\label{thm:3.6}
Let $\mathfrak{b}\subseteq \mathfrak{a}$ be ideals of $R$, $M$ a linearly compact $R$-module so that $\mathrm{Tor}_{i}^{R}(R/\mathfrak{a},M)$ is artinian for every integer $i$ and the set $\mathrm{Coass}_{R}\mathrm{H}_{\mathrm{r}_{\mathfrak{a}}^{\mathfrak{b}}(M)}^{\mathfrak{a}}(M)$ is finite. Then
\begin{center}
$r_{\mathfrak{a}}^{\mathfrak{b}}(M)=
\mathrm{inf}\{r_{\mathfrak{a}R_{\mathfrak{p}}}^{\mathfrak{b}R_{\mathfrak{p}}}(_{\mathfrak{p}}M) \hspace{0.03cm}|\hspace{0.03cm}\mathfrak{p}\in \mathrm{Spec}R\}.$
\end{center}
In particular, if $\mathfrak{a}=\mathfrak{b}$, then
\begin{center}
$r_{\mathfrak{a}}(M)=
\mathrm{inf}\{r_{\mathfrak{a}R_{\mathfrak{p}}}(_{\mathfrak{p}}M) \hspace{0.03cm}|\hspace{0.03cm}\mathfrak{p}\in \mathrm{Spec}R\}.$
\end{center}
\end{thm}

Nam \cite{N} defined an $R$-module $M$ is \emph{$\mathfrak{a}$-coartinian} if $\mathrm{Cosupp}_{R}M\subseteq \mathrm{V}(\mathfrak{a})$ and $\mathrm{Tor}_{i}^{R}(R/\mathfrak{a},M)$ is artinian for $i\geq 0$. Define $c_{\mathfrak{a}}^{\mathfrak{b}}(M):= \mathrm{inf}\{i\geq 0 \hspace{0.03cm}|\hspace{0.03cm}
\mathrm{H}_{i}^{\mathfrak{a}}(M)~\mathrm{is~not}~\mathfrak{b}$-$\mathrm{coartinian}\}$. In view of Theorem \ref{thm:3.6}, the next corollary provides another condition that makes the local-global principle for the artinianness dimension valid.

\begin{cor}\label{cor:4.3}
Let $\mathfrak{b}\subseteq \mathfrak{a}$ be ideals of $R$, $M$ a semi-discrete linearly compact $R$-module such that $\mathrm{Tor}_{i}^{R}(R/\mathfrak{a},M)$ is artinian for every integer $i$. If $r_{\mathfrak{a}}(M)\neq c_{\mathfrak{a}}^{\mathfrak{b}}(M)$, then
\begin{center}
$r_{\mathfrak{a}}^{\mathfrak{b}}(M)=
\mathrm{inf}\{r_{\mathfrak{a}R_{\mathfrak{p}}}^{\mathfrak{b}R_{\mathfrak{p}}}(_{\mathfrak{p}}M) \hspace{0.03cm}|\hspace{0.03cm}\mathfrak{p}\in \mathrm{Spec}R\}.$
\end{center}
\end{cor}
\begin{proof}
We first claim that
\begin{center}
$r_{\mathfrak{a}}(M)=\mathrm{min}\{r_{\mathfrak{a}}^{\mathfrak{b}}(M),c_{\mathfrak{a}}^{\mathfrak{b}}(M)\}$.
\end{center}
Set $s=r_{\mathfrak{a}}(M)$. Then $\mathfrak{b}\subseteq \mathfrak{a}\subseteq \sqrt{\mathrm{Ann}_{R}\mathrm{H}_{i}^{\mathfrak{a}}(M)}$ for $i< s$ and $r_{\mathfrak{a}}(M)\leq r_{\mathfrak{a}}^{\mathfrak{b}}(M)$. If $t=c_{\mathfrak{a}}^{\mathfrak{b}}(M)< r_{\mathfrak{a}}(M)$, then $\mathrm{H}_{t}^{\mathfrak{a}}(M)$ is an artinian $R$-module. Since $\mathrm{Cosupp}_{R}\mathrm{H}_{t}^{\mathfrak{a}}(M)\subseteq \mathrm{V}(\mathfrak{a})\subseteq \mathrm{V}(\mathfrak{b})$, it follows that $\mathrm{H}_{t}^{\mathfrak{a}}(M)$ is a $\mathfrak{b}$-coartinian $R$-module, which is a contradiction. Whence $r_{\mathfrak{a}}(M)\leq c_{\mathfrak{a}}^{\mathfrak{b}}(M)$ and so
$r_{\mathfrak{a}}(M)\leq \mathrm{min}\{r_{\mathfrak{a}}^{\mathfrak{b}}(M),c_{\mathfrak{a}}^{\mathfrak{b}}(M)\}$.
Now suppose that $s< \mathrm{min}\{r_{\mathfrak{a}}^{\mathfrak{b}}(M),c_{\mathfrak{a}}^{\mathfrak{b}}(M)\}$.
There exists an integer $n$ such that $\mathfrak{b}^{n}\mathrm{H}^{\mathfrak{a}}_{s}(M)=0$, thus $\mathrm{H}^{\mathfrak{a}}_{s}(M)\cong \mathrm{H}^{\mathfrak{a}}_{s}(M)/\mathfrak{b}^{n}\mathrm{H}^{\mathfrak{a}}_{s}(M)$. Since $s< c_{\mathfrak{a}}^{\mathfrak{b}}(M)$, it follows that $\mathrm{H}^{\mathfrak{a}}_{s}(M)$ is $\mathfrak{b}$-coartinian, $\mathrm{H}^{\mathfrak{a}}_{s}(M)/\mathfrak{b}\mathrm{H}^{\mathfrak{a}}_{s}(M)$ is artinian. The $R$-module $\mathrm{H}^{\mathfrak{a}}_{s}(M)$ is artinian from the above isomorphism, which is a contradiction. Hence $r_{\mathfrak{a}}(M)\geq \mathrm{min}\{r_{\mathfrak{a}}^{\mathfrak{b}}(M),c_{\mathfrak{a}}^{\mathfrak{b}}(M)\}$.
Now since $r_{\mathfrak{a}}(M)\neq c_{\mathfrak{a}}^{\mathfrak{b}}(M)$, it follows that $r_{\mathfrak{a}}(M)< c_{\mathfrak{a}}^{\mathfrak{b}}(M)$ and $r_{\mathfrak{a}}(M)=r_{\mathfrak{a}}^{\mathfrak{b}}(M)$. The assertion follows from Theorem \ref{thm:3.6} and \cite[Theorem 4.5]{CN2}.
\end{proof}

\bigskip
\section{\bf Modules with the set of coassociated primes finite}

In this section, we introduce a class $FM_{\leq n}$ of $R$-modules and prove that all semi-discrete linearly compact $\mathfrak{a}$-coartinian $R$-modules satisfy the local-global principle for the artinianness dimension.

\begin{df}\label{df:2.0}
Let $n$ be an integer.

$\mathrm{(1)}$ An $R$-module $M$ is said to be in $FM_{\leq n}$, if there exists a submodule $N$ of $M$ such that $\mathrm{mag}_{R}N\leq n$ and $M/N$ is artinian.

$\mathrm{(2)}$ Define
\begin{center}
$\mathrm{g}^{\mathfrak{a}}_{n}(M):=\mathrm{inf}\{i\geq0\hspace{0.03cm}|\hspace{0.03cm}
\mathrm{H}^{\mathfrak{a}}_{i}(M)\notin FM_{\leq n}\}$,
\end{center}
and adopt the convention that the infimum of the empty set of integers is to be taken as $\infty$.
\end{df}

\begin{rem}\label{rem:2.1}\rm
$\mathrm{(1)}$ $M \in FM_{\leq -1}$ if and only if $M$ is artinian.

$\mathrm{(2)}$ Following \cite{Z1}, an $R$-module $M$ is called minimax, if there exists a finitely generated submodule $N$ of $M$ such that $M/N$ is artinian. Minimax modules are in  $FM_{\leq 0}$. In particular, if $M$ is noetherian, artinian or semi-discrete linearly compact, then $M \in FM_{\leq 0}$.

$\mathrm{(3)}$ An $R$-module $M$ satisfies the finite condition
for coassocisted primes if the set of coassocisted primes of any submodule of $M$ is finite (see \cite{N3}). In this case, $\mathrm{mag}_{R}M\leq n$ for some non-negative integer $n$, thus $M \in FM_{\leq n}$.

$\mathrm{(4)}$ Following \cite{N2}, an $R$-module $M$ is called CFA if there is a submodule
$N$ such that $\mathrm{Cosupp}_{R}N$ is a finite set and $M/N$ is artinian. Hence $\mathrm{mag}_{R}M\leq n$ for some non-negative integer $n$, that is to say, CFA modules are in $FM_{\leq n}$.

$\mathrm{(5)}$ Any $R$-module with magnitude strictly less than $n$ is in $FM_{\leq n}$.
\end{rem}

Now we provide some basic properties of the class $FM_{\leq n}$.

\begin{lem}\label{lem:2.2}
Let $n$ be an integer. If $L\in FM_{\leq n}$, then $(\mathrm{Coass}_{R}L)_{> n}$ is finite.
\end{lem}
\begin{proof}
Since $L\in FM_{\leq n}$, there is a submodule $L'$ of $L$ such that $\mathrm{mag}_{R}L'\leq n$ and $L/L'$ is artinian. Hence $(\mathrm{Coass}_{R}L')_{> n}= \emptyset$ and $(\mathrm{Coass}_{R}L/L')_{> n}$ is finite. Now from the exact sequence
$0\rightarrow L'\rightarrow L\rightarrow L/L'\rightarrow 0$,
we obtain
\begin{center}$\begin{aligned}
(\mathrm{Coass}_{R}L)_{> n}
&\subseteq (\mathrm{Coass}_{R}L')_{>n}\cup (\mathrm{Coass}_{R}L/L')_{> n}\\
&\subseteq (\mathrm{Coass}_{R}L/L')_{> n}.
\end{aligned}$\end{center}
Thus the set $(\mathrm{Coass}_{R}L)_{> n}$ is finite.
\end{proof}

\begin{lem}\label{lem:2.3}
For any integer $n$, the class $FM_{\leq n}$ is a Serre subcategory of the category of $R$-modules.
\end{lem}
\begin{proof}
First, we see that the class of $R$-modules with magnitude strictly less than $n$ is a Serre subcategory of the category of $R$-modules. On the other hand, the class
of artinian $R$-modules is a Serre subcategory which is closed under injective
hulls. It follows from \cite[Corollary 3.5]{Y4} that $FM_{\leq n}$ is a Serre subcategory of the category of $R$-modules.
\end{proof}

\begin{lem}\label{lem:2.4} Let $n$ be an integer, $N$ a finitely generated $R$-module and $M\in FM_{\leq n}$. Then $\mathrm{Ext}_{R}^{i}(N,M)$ and $\mathrm{Tor}_{i}^{R}(N,M)$ are in $FM_{\leq n}$ for all $i\geq 0$.
\end{lem}
\begin{proof}
We only prove the claim for Tor modules, and the proof for
Ext modules is similar. Since $N$ is finitely
generated, it follows that $N$ possesses a free resolution
\begin{center}
$\mathbf{F}:\cdots\rightarrow F_{s}\rightarrow F_{s-1}\rightarrow \cdots\rightarrow F_{1}\rightarrow F_{0}\rightarrow 0$,
\end{center}
where $F_{i}$ is finitely generated free for $i\geq 0$. Thus $\mathrm{Tor}_{i}^{R}(N,M)=\mathrm{H}_{i}(\mathbf{F}\otimes_{R}M)$ is a subquotient of a direct sum of finitely many copies of $M$. The assertion follows from Lemma \ref{lem:2.3}.
\end{proof}

\begin{lem}\label{lem:2.5} Let $n$ be an integer, $N$ a finitely generated $R$-module and $M$ an arbitrary $R$-module. Suppose that $t$ is a non-negative integer such that $\mathrm{Tor}_{i}^{R}(N,M)\in FM_{\leq n}$ for $i\leq t$. Then for any finitely generated $R$-module $L$ with $\mathrm{Supp}_{R}L\subseteq \mathrm{Supp}_{R}N$, $\mathrm{Tor}_{i}^{R}(L,M)\in FM_{\leq n}$ for $i\leq t$.
\end{lem}
\begin{proof}
Since $\mathrm{Supp}_{R}L\subseteq \mathrm{Supp}_{R}N$, it follows from the Gruson's Theorem (cf. \cite[Lemma 2.2]{B}) that there exists a finite filtration
\begin{center}
$0=L_{0}\subset L_{1}\subset \cdots\subset L_{k}=L$,
\end{center}
such that the factors $L_{j}/L_{j-1}$ are homomorphic image of $N$ for $1\leq j\leq k$. Now consider the exact sequences
$$\xymatrix@C=10pt@R=5pt{0 \ar[r] & K \ar[r] & N \ar[r] & L_{1} \ar[r] & 0,\\
0 \ar[r] & L_{1} \ar[r] & L_{2} \ar[r] & L_{2}/L_{1} \ar[r] & 0,\\
         &              &            &     \vdots            & \\
0 \ar[r] & L_{k-1} \ar[r] & L_{k} \ar[r] & L_{k}/L_{k-1} \ar[r] & 0. }$$
From the long exact sequence
$$\cdots\rightarrow \mathrm{Tor}_{i}^{R}(L_{j-1},M)\rightarrow \mathrm{Tor}_{i}^{R}(L_{j},M)\rightarrow \mathrm{Tor}_{i}^{R}(L_{j}/L_{j-1},M)\rightarrow \mathrm{Tor}_{i-1}^{R}(L_{j}/L_{j-1},M)\rightarrow \cdots,$$
and induction on $k$, it suffices to prove the case when $k=1$. Thus there is an exact sequence
$$0\rightarrow K\rightarrow N\rightarrow L\rightarrow 0 \eqno(*)$$
for some finitely generated $R$-module $K$. Now, using induction on $t$. If $t=0$, then $L\otimes_{R}M$ is a quotient module of $N\otimes_{R}M$, in view of assumption and Lemma \ref{lem:2.3}, $L\otimes_{R}M\in FM_{\leq n}$. Assume that $t> 0$ and $\mathrm{Tor}_{j}^{R}(L',M)\in FM_{\leq n}$ for every finitely generated $R$-module $L'$ with $\mathrm{Supp}_{R}L'\subseteq \mathrm{Supp}_{R}M$ and $j\leq t-1$. The exact sequence $(*)$ induces the long exact sequence
$$\cdots\rightarrow \mathrm{Tor}_{i}^{R}(N,M)\rightarrow \mathrm{Tor}_{i}^{R}(L,M)\rightarrow \mathrm{Tor}_{i-1}^{R}(K,M)\rightarrow \cdots.$$
By assumption and the inductive hyphothesis, $\mathrm{Tor}_{i}^{R}(N,M)$ and $\mathrm{Tor}_{i-1}^{R}(K,M)$ are in $FM_{\leq n}$ for $i\leq t$. It follows from Lemma \ref{lem:2.3} that $\mathrm{Tor}_{i}^{R}(L,M)\in FM_{\leq n}$ for $i\leq t$.
\end{proof}

Nam \cite{N1} posed a question on local homology: when is the set of coassociated primes of local homology modules finite? The following theorem is a partial answer to this question.

\begin{thm}\label{thm:2.10}
Let $M$ be a semi-discrete linearly compact $R$-module and $t=\mathrm{g}^{\mathfrak{a}}_{n}(M)$. For an integer $n$, the following statements hold:

$\mathrm{(1)}$ $\mathrm{Tor}_{j}^{R}(R/\mathfrak{a},\mathrm{H}_{i}^{\mathfrak{a}}(M))\in FM_{\leq n}$ for $i=0,1,\ldots,t-1$ and $j\geq 0$.

$\mathrm{(2)}$ $R/\mathfrak{a}\otimes_{R}\mathrm{H}_{t}^{\mathfrak{a}}(M)$ and $\mathrm{Tor}_{1}^{R}(R/\mathfrak{a},\mathrm{H}_{t}^{\mathfrak{a}}(M))\in FM_{\leq n}$.

$\mathrm{(3)}$ For each finitely generated $R$-module $N$ with $\mathrm{Supp}_{R}N\subseteq \mathrm{V}(\mathfrak{a})$, $\mathrm{Tor}_{j}^{R}(N,\mathrm{H}_{i}^{\mathfrak{a}}(M))\in FM_{\leq n}$ for $i=0,1,\ldots,t-1$ and $j\geq 0$.

$\mathrm{(4)}$ The set $(\mathrm{Coass}_{R}\mathrm{H}_{t}^{\mathfrak{a}}(M))_{> n}$ is finite.
\end{thm}
\begin{proof}
$\mathrm{(1)}$ Since $t=\mathrm{g}^{\mathfrak{a}}_{n}(M)$, it yields that $\mathrm{H}_{i}^{\mathfrak{a}}(M)\in FM_{\leq n}$ for $i=0,1,\ldots,t-1$. The result follows immediately from Lemma \ref{lem:2.4}.

$\mathrm{(2)}$ We have $M\in FM_{\leq n}$ because $M$ is a semi-discrete linearly compact $R$-module. Then $\mathrm{Tor}_{t}^{R}(R/\mathfrak{a},M)$ and $\mathrm{Tor}_{t+1}^{R}(R/\mathfrak{a},M)\in FM_{\leq n}$. Using Lemma \ref{lem:2.7} and part $\mathrm{(1)}$, the assertion holds true.

$\mathrm{(3)}$ It follows from Lemma \ref{lem:2.5} and part $\mathrm{(1)}$.

$\mathrm{(4)}$ Note $\mathrm{Coass}_{R}(R/\mathfrak{a}\otimes_{R}\mathrm{H}_{t}^{\mathfrak{a}}(M))
=\mathrm{Coass}_{R}\mathrm{H}_{t}^{\mathfrak{a}}(M)$. The assertion follows from Lemma \ref{lem:2.2} and part $\mathrm{(2)}$.
\end{proof}

The following corollary provide some conditions such that the local-global principle for the artinianness dimension holds.

\begin{cor}\label{cor:3.8}
Let $M$ be a linearly compact $R$-module such that $\mathrm{Tor}_{i}^{R}(R/\mathfrak{a},M)$ is artinian for every integer $i$. Then
\begin{center}
$r_{\mathfrak{a}}(M)=
\mathrm{inf}\{r_{\mathfrak{a}R_{\mathfrak{p}}}(_{\mathfrak{p}}M) \hspace{0.03cm}|\hspace{0.03cm}\mathfrak{p}\in \mathrm{Spec}R\},$
\end{center}
if one of the following conditions is satisfied:

$\mathrm{(1)}$ $M$ is semi-discrete linearly compact $R$-modue;

$\mathrm{(2)}$ $M$ and $\mathrm{H}_{i}^{\mathfrak{a}}(M)$ satisfy the finite condition for coassocisted primes for all $i< r_{\mathfrak{a}}(M)$;

$\mathrm{(3)}$ $M$ and $\mathrm{H}_{i}^{\mathfrak{a}}(M)$ are CFA for all $i<r_{\mathfrak{a}}(M)$;
\end{cor}

\begin{proof}
Note that $\mathrm{g}^{\mathfrak{a}}_{-1}(M)=r_{\mathfrak{a}}(M)$. By Theorem \ref{thm:2.10}(4), \cite[Theprorem 4.5]{CN2}, \cite[Theorem 3.1]{N3} and \cite[Theorem 1]{N2}, the set $\mathrm{Coass}_{R}\mathrm{H}_{r_{\mathfrak{a}}(M)}^{\mathfrak{a}}(M)$ is finite. Hence the assertion follows by Theorem \ref{thm:3.6}.
\end{proof}

Part $\mathrm{(1)}$ of above corollary indicates that semi-discrete linearly compact $\mathfrak{a}$-coartinian modules must satisfiy the local-global principle for the artinianness dimension and no longer need the condition in Corollary \ref{cor:4.3}.


\bigskip
\begin{thebibliography}{99}
\bibitem{B} K. Bahmanpour. Cohomological dimension, cofiniteness and abelian categories of cofinite modules. \emph{J. Algebra} 2017, \textbf{484}: 168--197.
\bibitem{B1} M. P. Brodmann, R. Y. Sharp. \emph{Local cohomology: an algebraic introduction with geometric applications}. Cambridge University Press, 1998.
\bibitem{CN2} N. T. Cuong, T. T. Nam. On the co-localization, co-support and co-associated primes of local homology modules. \emph{Vietnam J. Math.} 2001, \textbf{29}(4): 359--368.
\bibitem{CN0} N. T. Cuong, T. T. Nam. The $I$-adic completion and local homology for artinian modules. \emph{Math. Proc. Cambridge Philos. Soc.} 2001, \textbf{131}(1): 61--72.
\bibitem{CN1} N. T. Cuong, T. T. Nam. A local homology theory for linearly compact modules. \emph{J. Algebra} 2008, \textbf{319}(11): 4712--4737.
\bibitem{F0} G. Faltings. Der Endlichkeitssatz in der lokalen Kohomologie. \emph{Math. Ann.}
1981, \textbf{255}(1): 45--56.
\bibitem{M} I. G. Macdonald. Duality over complete local rings. \emph{Topology} 1962, \textbf{1}(3): 213--235.
\bibitem{M4} E. Matlis. The higher properties of $R$-sequences. \emph{J. Algebra}  1978, \textbf{50}(1): 77--112.
\bibitem{N3} T. T. Nam. On the finiteness of co-associated primes of local homology modules. \emph{J. Math. Kyoto Univ.} 2008, \textbf{48}(3): 521--527.
\bibitem{N} T. T. Nam. Co-support and coartinian modules. \emph{Algebra Colloq.} 2008, \textbf{15}(1): 83--96.
\bibitem{N1} T. T. Nam, Minimax modules, local homology and local cohomology, \emph{Internat. J. Math.} \textbf{26} (2015) 16pp.
\bibitem{N2} M. T. Nguyen.  CFA modules and the finiteness of coassociated primes of local homology modules. \emph{Hiroshima Math. J.} 2021, \textbf{51}(2): 155--161.
\bibitem{Y1} S. Yassemi. Coassociated primes. \emph{Comm. Algebra} 1995, \textbf{23}(4): 1473--1498.
\bibitem{Y2} S. Yassemi. Magnitude of modules. \emph{Comm. Algebra} 1995, \textbf{23}(11): 3993--4008.
\bibitem{Y4} T. Yoshizawa. Subcategories of extension modules by Serre subcategories. \emph{Proc. Amer. Math. Soc.} 2012, \textbf{140}(7): 2293--2205.
\bibitem{Z1} H. Z$\ddot{\textrm{o}}$schinger. Minimax modules. \emph{J. Algebra} 1986, \textbf{102}(1): 1--32.
\end{thebibliography}
\end{document}